\documentclass[reqno,b5paper]{amsart}
\usepackage{amsmath}
\usepackage{amssymb}
\usepackage{amsthm}
\usepackage{enumerate}
\usepackage[mathscr]{eucal}
\theoremstyle{plain}
\newtheorem{thm}{Theorem}[section]

\newtheorem{Example}{Example}[section]
\newtheorem{note}{Note}[section]

\theoremstyle{definition}
\newtheorem{defn}{Definition}[section]

\begin{document}

\setcounter {page}{1}
%---------------Title,Author,Abstract-----------------------------------------------
\title{A note on rough $I$-convergence of double sequences}

\author[P. Malik, M. Maity AND A. GHOSH]{ Prasanta Malik*, Manojit Maity** and Argha Ghosh*\ }
\newcommand{\acr}{\newline\indent}
\maketitle
\address{{*\,} Department of Mathematics, The University of Burdwan, Golapbag, Burdwan-713104,
West Bengal, India.
                Email: pmjupm@yahoo.co.in., buagbu@yahoo.co.in\acr
           {**\,} Boral High School, Kolkata-700154, India. Email: mepsilon@gmail.com\\}

\maketitle
\begin{abstract}
In this paper we study some basic properties of rough $I$-convergent double sequences in the line of D$\ddot{u}$ndar [8]. We also study the set of all rough $I$-limits of a double sequence and relation between boundedness and rough $I$-convergence of a double sequence.
\end{abstract}
\author{}
\maketitle


%-------------------------Section 1- Background and introduction-----------------------
\section{\textbf{Introduction:}}
The concept of $I$-convergence of double sequences was introduced by Balcerzak et. al. [2]. The notion of $I$-convergence of a double sequence, which is based on the structure of the ideal $I$ of subsets of $\mathbb{N} \times \mathbb{N} $, where $\mathbb{N}$ is the set of all natural numbers, is a natural generalization of the notion of convergence of a double sequence in Pringsheim's sense [17] as well as the notion of statistical convergence of a double sequence [14].

A lot of work on $I$-convergence of double sequences can be found in ([3], [4], [5], [7] etc.) and many others.

The concept of rough $I$-convergence of single sequences was introduced by Pal et. al. [15] which is a generalization of the earlier concepts namely rough convergence [16] and rough statistical convergence [1] of single sequences. Recently rough statistical convergence of double sequences has been introduced by Malik and Maity [13] as a generalization of rough convergence of double sequences [12] and investigated some basic properties of this type of convergence and also studied relation between the set of statistical cluster points and the set of rough limit points of a double sequence. Recently the notion of rough $I$-convergence for double sequences has been introduced by D$\ddot{u}$ndar [8]. In this paper we investigate some basic properties of rough $I$-convergence of double sequences in finite dimensional normed linear spaces which are not done earlier. We study the set of rough $I$-limits of a double sequence and also the relation between boundedness and rough $I$-convergence of a double sequence.

%------------------------------Section-2 - Basic definitions-------------------------
\section{\textbf{Basic Definitions and Notations}}
%------------------------------------introduction to section2-----------------------
%---------------------------Definition 2.1------------------------------------------
Throughout the paper $\mathbb{N}$ denotes the set of all positive integers and $\mathbb{R}$ denotes the set of all real numbers.
\begin{defn}[12]
Let $x=\{x_{jk}\}_{j,k~\in\mathbb{N}}$ be a double sequence in a
normed linear space$(X,\parallel.\parallel)$ and $r$ be a
non-negative real number. $x$ is said to be $r$-convergent to $\xi
\in X $, denoted by $x \overset {r}\rightarrow \xi$, if for any
${\epsilon}>0$ there exists $N_{\epsilon} ~ \in~ \mathbb{N}$ such
that for all $j,k\geq N_{\epsilon}$ we have
\begin{center}
$\parallel {x_{jk}-\xi}\parallel~< r+{\epsilon}$.
\end{center}
\end{defn}
In this case $\xi$ is called an $r$-limit of $x$.

It is clear that rough limit of $x$ is not necessarily unique (for $r>0$). So we consider
$r$-limit set of $x$ which is denoted by ${LIM}_x ^r$ and is
defined by ${LIM}_x ^r = \{\xi\in X: x \overset {r} \rightarrow
\xi\}$. $x$ is said to be $r$-convergent if ${LIM}_x ^r
\neq \emptyset$ and $r$ is called a rough convergence degree of $x$.

We recall that a subset $ K$ of $ \mathbb{N} \times \mathbb{N} $ is said to have natural density $d(K)$ if
\begin{center}
$d(K) = \lim\limits_{\stackrel{\stackrel{m\rightarrow
\infty}{n\rightarrow \infty}} ~} \frac{ K(n,m)}{n.m} $,
\end{center}
where $K(n,m) = |\{(j,k) \in \mathbb{N} \times \mathbb{N}: j \leq n, k \leq m \}| $.
%---------------------------Definition 2.2------------------------------------------
\begin{defn}[13]
Let $x = \{x_{jk}\}_{j,k \in \mathbb{N}}$ be a double sequence in a
normed linear space $(X, \parallel.\parallel)$ and $r$ be a non negative real number.
$x$ is said to be $r$- statistically convergent to $\xi$, denoted by $x \overset{r-st_2}\longrightarrow \xi$,
if for any $\varepsilon > 0$ we have $d(A(\varepsilon)) = 0$, where
$A(\varepsilon) = \{ (j,k) \in \mathbb{N}\times \mathbb{N} : \parallel x_{jk} - \xi \parallel \geq r + \varepsilon \}$.
In this case $\xi$ is called $r$-statistical limit of $x$.
\end{defn}
Clearly for $r = 0$ from Definition 2.1 we get Pringsheim convergence of double sequences and from Definition 2.2 we get ordinary statistical convergence of double sequences.

%---------------------------Definition 2.3------------------------------------------
\begin{defn}
A class $I$ of subsets
of a nonempty set $X$ is said to be an ideal in $X$ provided

(i) $\phi\in I$.

(ii) $ A,B\in I$ $~$implies $A\bigcup B\in I$.

(iii) $ A\in I,B\subset A $$~$ implies $~~$   $B\in I$.

$I$ is called a nontrivial ideal if $X\notin I$.

\end{defn}
%---------------------------Definition 2.4------------------------------------------
\begin{defn}
A non empty class $F$ of
subsets of a nonempty set $X$ is said to be a filter in $X$ provided

(i) $\phi\notin F$.

(ii) $A,B\in F$ $~$ implies $~~$ $A\bigcap B\in F$.

(iii) $A\in F,A\subset B$ $~$ implies $~~$ $B\in F$.

If $I$ is a nontrivial ideal in $X$, $X\neq\phi$, then the class
\begin{center}
$F(I)=\{ M \subset X  : M = X \setminus A$ for some $A \in I \}$
\end{center}
is a filter on $X$, called the filter associated with $I$.
\end{defn}
%---------------------------Definition 2.5------------------------------------------
\begin{defn}[4]
A nontrivial ideal $I$ in $X$ is called admissible if $\{x\} \in I$ for each $x \in X $.
\end{defn}
%---------------------------Definition 2.6------------------------------------------
\begin{defn}[4]
A nontrivial ideal $I$ on $\mathbb{N} \times \mathbb{N} $ is called strongly admissible if $\left\{i\right\}\times \mathbb{N}$ and $\mathbb{N}\times\left\{i\right\}$ belong to $I$ for each $i\in \mathbb{N}$.
\end{defn}
Clearly every strongly admissible ideal is admissible.
Throughout the paper we take $I$ as a strongly admissible ideal in $\mathbb{N} \times \mathbb{N} $.

%---------------------------Definition 2.7------------------------------------------
\begin{defn} [8]
Let $ x = \{x_{jk}\}_{j,k \in \mathbb{N}} $ be a double sequence in a normed linear space $ (X, \parallel . \parallel) $ and $ r $ be a non negative real number. Then $ x $ is said to be rough ideal convergent or $ rI$-convergent to $ \xi $, denoted by $ x \overset{rI}\longrightarrow \xi $, if for any $ \varepsilon > 0 $ we have $ \{(j,k) \in \mathbb{N} \times \mathbb{N} :\parallel x_{jk} - \xi \parallel \geq r + \varepsilon \} \in I $. In this case $ \xi $ is called $ rI $-limit of $ x $ and $x$ is called rough $I$-convergent to $\xi$ with $r$ as roughness degree.
\end{defn}
Throughout this paper $ x $ denotes the double sequence $\{x_{jk}\}_{j,k \in \mathbb{N}}$  in a normed linear space$(X, \parallel . \parallel)$ and $r$ denotes a non negative real number.

For $ r = 0 $ we get the usual $I$-convergence of double sequences. But our main interest is on the case where $ r > 0 $. Because it may  happen that a double sequence $ x =\{x_{jk}\}_{j,k \in \mathbb{N}}$ is not $I$-convergent in usual sense but there exists a double sequence $ y = \{y_{jk}\}_{j,k \in \mathbb{N}} $ which is $I$-convergent in usual sense and $ \parallel x_{jk} - y_{jk} \parallel \leq r $ for all $ (j,k) \in \mathbb{N} \times \mathbb{N} $ (or $ \{ (j,k) \in \mathbb{N} \times \mathbb{N} : ~\parallel x_{jk} - y_{jk} \parallel > r \} \in I $) for some $ r > 0 $. Then $ x $ is $ rI$-convergent.

From the definition it is clear that $ rI$-limit of $ x $ is not necessarily unique (for $ r > 0 $). So we consider $ rI $-limit set of $ x $, which is denoted by $ I-LIM^r_x = \{\xi \in X : x \overset{rI}\longrightarrow \xi \}$. $x$ is said to be $ rI $-convergent if $I-LIM^r_x \neq \emptyset $ and $ r $ is called a rough $I$-convergence degree of $ x $.

%-----------------------Defn 2.8---------------------------------------------------------------------

\begin{defn}
A double sequence $ x $ in $ X $ is said to be bounded if there exists a positive real number $ M $ such that $ \parallel x_{jk} \parallel < M $ for all $(j,k) \in \mathbb{N} \times \mathbb{N}$.
\end{defn}

%----------------------Defn 2.9-------------------------------------------------------------------------

\begin{defn}
A double sequence $ x $ in $ X $ is said to be $I$-bounded if there exists a positive real number $ M $ such that $\{(j,k) \in \mathbb{N} \times \mathbb{N} : \parallel x_{jk} \parallel \geq M \} \in I $.
\end{defn}

%----------------------Defn 2.10-------------------------------------------------------------------------

\begin{defn}[5]
A point $\xi \in X $ is said to be an $I$-cluster point of a double sequence $x = \{x_{jk}\}_{j, k \in \mathbb{N}}$ if and only if for each $\varepsilon > 0$ the set $ \{(j,k)\in \mathbb{N} \times \mathbb{N} : \parallel x_{jk} - \xi \parallel < \varepsilon \} \notin I$. We denote the set of all $I$-cluster points of $x$ by $I(\Gamma_x)$.
\end{defn}
%----------------------Theorm 2.1-------------------------------------------------------------------------
\begin{thm}[5]
An $I$-bounded double sequence $ x=
\{x_{jk}\}_{j,k \in \mathbb{N} }$ of real numbers is $I$-convergent if and only if
$I-\limsup x = I-\liminf x.$
\end{thm}
%----------------------Theorm 2.2-------------------------------------------------------------------------
\begin{thm}[5]
Let $x= \{x_{jk}\}_{j,k \in \mathbb{N} }$ be a
bounded double sequence of real numbers, then

(i) $I-\limsup x = \max~I(\Gamma_x)$,

(ii) $I-\liminf x = \min~ I(\Gamma_x).$

\end{thm}
The above result is also true for $I$-bounded double sequences. So it can be stated as follows.
%----------------------Theorm 2.3-------------------------------------------------------------------------

\begin{thm}
Let $x= \{x_{jk}\}_{j,k \in \mathbb{N} }$ be an
$I$-bounded double sequence of real numbers, then

(i) $I-\limsup x = \max~I(\Gamma_x)$,

(ii) $I-\liminf x = \min~ I(\Gamma_x).$

\end{thm}

\begin{thm} [8]
For a double sequence $x = \{ x_{jk}\}_{ j,k \in \mathbb{N}}$ in a
normed linear space $( X, \parallel . \parallel )$ we have $diam(I-LIM_x^r) \leq 2r $.
In particular if $x \overset{I}\longrightarrow \xi $, then $I-LIM_x^r =
\overline{B_r}(\xi) = \{ y \in X: \parallel y - \xi \parallel \leq r\}$ and so
$ diam(I-LIM_x^r) = 2r $.
\end{thm}

%-------------------------------------------Note 2.1 ------------------------------------------------------

\begin{note}
 When r=0, then $diam\left(I-LIM_x^r\right)=0$. Therefore $I-LIM^r_x$ is either $\phi$ or singleton. This implies the uniqueness of limit of $I$-convergent double sequence.
\end{note}

\begin{thm} [8]
Let $ x = \{x_{jk}\}_{j, k \in \mathbb{N}}$ be a double sequence in $X$ and $c \in I(\Gamma_x)$. Then $ \parallel \xi - c \parallel \leq r $ for all $ \xi \in I-LIM_x^r $ i.e. $I-LIM_x^r \subset \overline{B_r}(c) $. 
\end{thm}

We now consider an example of a double sequence which is rough $I$-convergent but not rough convergent.

%-----------------------------Example 2.1 --------------------------------------------
\begin{Example}
We consider the ideal $I_{d} = \{ A \subset \mathbb{N} \times \mathbb{N}: d(A) = 0 \}$. Let $ x =\{x_{jk}\}_{j,k \in \mathbb{N}}$ be a double sequence in the normed linear space $(\mathbb{R} , \parallel .\parallel)$ defined by
\begin{eqnarray*}
 x_{jk} &=& 2jk ~~~~,\mbox{if}~~~j ~~~\mbox{and}~~~k ~\mbox{are squares},\\
        &=& (-1)^{j+k} ,~\mbox{otherwise}.
\end{eqnarray*}
Then
\[ I_{d}-LIM_x^r = \left\{
  \begin{array}{l l}
    \emptyset & ;\quad \text{if $r<1$ }\\
    \left[1-r,r-1\right] &; \quad \text{if $r \geq 1$ }
  \end{array} \right.\]\\
and $ LIM_x^r = \emptyset $ for all $ r \geq 0$.
\end{Example}
From the above example we see that $I-LIM_x^r  \neq \emptyset $ does not imply $ LIM_x^r \neq \emptyset $. But $ LIM_x^r \neq  \emptyset $ always implies that $I-LIM_x^r  \neq \emptyset $.

%--------------------------Sectin-3 (mian Results)---------------------------------------------------

\section{\textbf{Main Results}}

%In this section we study  some basic properties of rough $I$-convergence of double sequences.

We first establish a relation between boundedness and rough $I$-convergence of double sequences.
%-------------------------------------------Theorem 3.1 ------------------------------------------------------
\begin{thm}
 If a double sequence $x=\left\{x_{jk}\right\}$ is bounded, then there exists $r\geq 0$ such that $I-LIM_x^r\neq \phi$.
\end{thm}
\begin{proof}
The proof is similar to the proof of Theorem 3.2 [13], so is omitted.
\end{proof}
%-------------------------------------------Note 3.2 ------------------------------------------------------
\begin{note}
Taking $I=\left\{A\in \mathbb{N}\times\mathbb{N}: d\left(A\right)= 0\right\}$, from Note 3.2 [13] we see that  the converse of Theorem 3.1 is not true.
 \end{note}
We now show that the converse of Theorem 3.1 is true if the double sequence $x$ is $I$-bounded.
%--------------------------------------------Theorem 3.2 -------------------------------------------------

\begin{thm}
A double sequence $ x $ is $I$-bounded if and only if there exists $ r \geq 0 $ such that $ I-LIM^r_x \neq \emptyset $.
\end{thm}

\begin{proof}
Let $ x $ be an $I$-bounded double sequence. Then there exists a positive real number $ M $ such that $ A = \{(j,k) \in \mathbb{N} \times \mathbb{N} : \parallel x_{jk} \parallel \geq M \} \in I $. Let $ r^{'} = sup \{ \parallel x_{jk} \parallel: (j,k) \in \mathbb{N} \times \mathbb{N} \setminus A \}$. Then $ 0 \in I-LIM_x^{r^{'}} $ and so $ I-LIM_x^{r^{'}} \neq \emptyset $.\\

Conversely, let $ I-LIM_x^r \neq \emptyset $ for some $ r \geq 0 $. Let $ \xi \in I-LIM_x^r $. Take $\varepsilon = 1 $. Then $ B =  \{(j,k) \in \mathbb{N} \times \mathbb{N} : \parallel x_{jk} - \xi \parallel \geq 1 + r \} \in I $. Now $\{(j,k) \in \mathbb{N} \times \mathbb{N}: \parallel x_{jk} \parallel \geq 1+r + \parallel \xi \parallel \} \subset B $ and so $\{(j,k) \in \mathbb{N} \times \mathbb{N}: \parallel x_{jk} \parallel \geq 1 + r + \parallel \xi \parallel \} \in I $. This shows that $ x $ is $I$-bounded.
\end{proof}

Next we present an alternative proof of Theorem 2.4 [8] which gives a topological property of the $rI$-limit set of a double sequence..

%----------------------Theorem 3.3 --------------------------------------------------------------

\begin{thm}
For all $ r \geq 0 $, the $rI$-limit set $ I-LIM^r_x $, of a double sequence $ x = \{x_{jk}\}_{j, k \in \mathbb{N}} $ is closed.
\end{thm}

\begin{proof}
Let $ \xi $ be a limit point of $ I-LIM_x^r $. Then for any $ \varepsilon > 0 $, $ B_{\frac{\varepsilon}{2}}(\xi) \bigcap I-LIM_x^r \neq \emptyset $. Let $ \alpha \in B_{\frac{\varepsilon}{2}}(\xi) \bigcap I-LIM^r_x$. Since $ \alpha \in I-LIM^r_x$ so $A(\frac{\varepsilon}{2}) = \{(j,k) \in \mathbb{N} \times \mathbb{N} : \parallel x_{jk} - \alpha \parallel \geq r + \frac{\varepsilon}{2}\} \in I $. Let $B(\varepsilon) = \{(j,k) \in \mathbb{N} \times \mathbb{N} : \parallel x_{jk} - \xi \parallel \geq r + \varepsilon\}$.\\
Now $(j,k) \notin A(\frac{\varepsilon}{2})$ implies $ (j,k) \notin B(\varepsilon) $. Thus $(j,k) \in B(\varepsilon)$ implies $(j,k) \in A(\frac{\varepsilon}{2})$. This implies $B(\varepsilon) \subset A(\frac{\varepsilon}{2})$ and so $B(\varepsilon) = \{(j,k) \in \mathbb{N} \times \mathbb{N} : \parallel x_{jk} - \xi \parallel \geq r + \varepsilon \} \in I $. Therefore $ \xi \in I-LIM^r_x$. Hence $ I-LIM^r_x$ is a closed set in $ X $.
\end{proof}

%--------------------------------Theorem 3.4----------------------------------------------------

\begin{thm}
Let $x = \{x_{jk}\}_{ j, k \in \mathbb{N}} $ be a double sequence in $X$. Then $x$ is $I$-convergent to $ \xi $ if and only if $I-LIM^r_x= \overline{B_r}(\xi)$.
\end{thm}

\begin{proof}
It directly follows from Theorem 2.4 that if $ x $ is $I$-convergent to $ \xi $, then $I-LIM^r_x= \overline{B_r}(\xi)$.

Conversely, let $I-LIM^r_x= \overline{B_r}(\xi)$. We have to show that $ x $ is $I$-convergent to $ \xi $, i.e. for all $ a > 0 $, $A(a)= \{(j,k) \in \mathbb{N} \times \mathbb{N} : \parallel x_{jk} - \xi \parallel \geq a \} \in I $. Now fixed $ a > 0 $. Let us choose $ r > 0 $ and $ \varepsilon > 0 $ such that $ r + \varepsilon < a $. For $ \xi \in I-LIM^r_x$,  $\{(j,k) \in \mathbb{N} \times \mathbb{N} : \parallel x_{jk} - \xi \parallel \geq r + \varepsilon\} \in I $. Since $\{(j,k) \in \mathbb{N} \times \mathbb{N} : \parallel x_{jk} - \xi \parallel \geq a\} \subset \{(j,k) \in \mathbb{N} \times \mathbb{N} : \parallel x_{jk} - \xi \parallel \geq r + \varepsilon \}$. So $\{(j,k) \mathbb{N} \times \mathbb{N} : \parallel x_{jk} -\xi \parallel \geq a\} \in I $. Hence $ x $ is $I$-convergent to $ \xi $.
\end{proof}

%-----------------------------------theorem 3.8 -----------------------------------------------------------------

\begin{thm}
Let $(\mathbb{R}, \parallel . \parallel)$ be a strictly convex space and $ x = \{x_{jk}\}_{ j, k \in \mathbb{N}} $ be double sequence in $\mathbb{R}$. For any $ r > 0 $, let $ y_1, y_2 \in I-LIM^r_x$ with $\parallel y_1 - y_2 \parallel = 2r $. Then $ x $ is $I$-convergent to $ \frac{1}{2}(y_1 + y_2)$.
\end{thm}

\begin{proof}
Let $ y_3 $ be an arbitrary $I$-cluster point of $ x $. Now since $ y_1, y_2 \in I-LIM^r_x$, so by Theorem 2.5 we have
\begin{center}
 $ \parallel y_1 - y_3 \parallel \leq r $ and $ \parallel y_2 - y_3 \parallel \leq r $.
\end{center}
Then $ 2r = \parallel y_1 - y_2 \parallel \leq \parallel y_1 - y_3 \parallel + \parallel y_3 - y_2 \parallel \leq 2r $. Therefore $\parallel y_1 - y_3 \parallel = \parallel y_2 - y_3 \parallel = r $. Now
\begin{eqnarray}
\frac{1}{2}(y_1 - y_2) = \frac{1}{2}[(y_3 - y_1)+(y_2 - y_3)].
\end{eqnarray}
Since $\parallel y_1 - y_2 \parallel = 2r $, so $\frac{1}{2} \parallel y_2 - y_1 \parallel = r $. Again since the space is strictly convex, so by (1) we get $\frac{1}{2}(y_2 - y_1) = y_3 - y_1 = y_2 - y_3 $. Thus $ y_3 $ is the unique $I$-cluster point of the double sequence $ x $. Again by the given condition $ I-LIM^r_x\neq \emptyset $, so by  Theorem 3.2 $ x $ is $I$-bounded. Since $ y_3 $ is the unique $I$-cluster point of the $I$-bounded double sequence $ x $, so by Theorem 2.1 and Theorem 2.3 $ x $ is $I$-convergent to $ y_3 = \frac{1}{2}( y_1 + y_2 )$.
\end{proof}

\noindent\textbf{Acknowledgement:} The authors are grateful to
Prof. Pratulananda Das, Department of Mathematics, Jadavpur
University for his advice during the preparation of this paper.
\\
%-------------------------------References----------------------------------------------------------------------

\end{document}